\documentclass{amsart}
\usepackage{graphics}
\usepackage{ae}
\usepackage{amsmath,bbm}

\usepackage{amssymb}

\def\1{{\mathbf 1}}

\numberwithin{equation}{section}

\newtheorem{theo}{Theorem}
\newtheorem{prop}[theo]{Proposition}

\newtheorem{coro}[theo]{Corollary}
\newtheorem{lemma}[theo]{Lemma}
\theoremstyle{definition}

\begin{document}

\title[HD$0$L $\omega$-equivalence and periodicity problems in the primitive case]{HD$0$L $\omega$-equivalence and periodicity problems in the primitive case (To the memory of G. Rauzy)}
\author{Fabien Durand}
\address[F.D.]{\newline
Universit\'e de Picardie Jules Verne\newline
Laboratoire Ami\'enois de Math\'ematiques Fondamentales et
Appliqu\'ees\newline
CNRS-UMR 6140\newline
33 rue Saint Leu\newline
80039 Amiens Cedex 01\newline
France.}
\email{fabien.durand@u-picardie.fr}

\begin{abstract}
In this paper I would like to witness the mathematical inventiveness of G. Rauzy through a personnal exchanges I had with him. 
The objects that will emerge will be used to treat the decidability of the HD$0$L $\omega$-equivalence and periodicity problems in the primitive case.
\end{abstract}

\maketitle

\section{Introduction}

\subsection{Some words about G\'erard Rauzy}

I started my PhD in 1992 under the direction of G. Rauzy. 
The topic of my PhD was {\em Substitutions and Bratteli diagrams}. 
G. Rauzy was a worldwide recognized expert on substitutions (and many other topics) but Bratteli diagram representations of substitution subshifts was not one of these subjects.
Thus I had an unformal second advisor, B. Host. 
Later, as a common agreement, he became my advisor.

Some days after I begun my PhD he told me that he noticed something amazing about substitutions.
He wrote on the blackboard the Fibonacci fixed point on the alphabet $\{ 0,1 \}$ starting with $0$.
He considered all occurrences of the letter $0$ and coded the words between two consecutive such occurrences. 
It is not surprising (see \cite{Morse&Hedlund:1940}) that only two words appear : $01$ and $0$.
If you respectively code them by $0$ and $1$ then you again obtain the same sequence.
This is not surprising too because this sequence is the fixed point of the morphism defined by $0\mapsto 01$ and $1\mapsto 0$.
Then he wrote a long prefix of the fixed point of the Morse sequence that we call $x$. 
Applying the same coding process (see below) to $x$ it clearly appeared that the new sequence called the {\em derived sequence of $x$}, and denoted by $\mathcal{D} (x)$, is not the same for the obvious reason that $3$ different letters occur.
He repeated the process to this new sequence.
Let's do it. 
The result was a sequence on a $4$ letter alphabet. 
Once again he repeated the process.
Amazingly we obtained the same sequence (the same long prefix).
Repeating the process will clearly produce the same sequence. 
The process is periodic. 

$$
\begin{array}{lllllllllllllllll}
x &= 011 & 01 & 0 & 011 & 0 & 01 & 011 & 01 & 0 & 01 & 011 & 0 & 011 & 01 & 0 \dots \\
\mathcal{D} (x) & = 0   & 1  & 2 & 0   & 2 & 1  & 0   & 1  & 2 & 1  & 0   & 2 & 0   & 1  & 2 \dots \\
\end{array}
$$

$$
\begin{array}{llllllllllllllllllll}
\mathcal{D} (x)   & = 012 & 021 & 0121 & 02 & 012 & 021 & 02 & 0121 & 012 & 021 & 0121 & 02  \dots \\
\mathcal{D}^2 (x) & = 0   & 1   & 2    & 3  & 0   & 1   & 3  & 2    & 0   & 1   & 2    & 3   \dots
\end{array}
$$

$$
\begin{array}{llllllllllllllllllll}
\mathcal{D}^2 (x) & = 0123 & 0132 & 01232 & 013 & 0123 & 0132 & 013 & 01232 & 0123 & 0132  \dots \\
\mathcal{D}^3 (x) & = 0    & 1    & 2     & 3   & 0    & 1    & 3   & 2     & 0    & 1\dots
\end{array}
$$

Then he showed me another example where the period was 2. 
He was wondering if such a period always exists for fixed points of primitive substitutions. 
He knew that when this process is eventually periodic then the sequence is primitive substitutive and he explained me the proof on the blackboard. 
 
This proof was very useful for my PhD and for the present paper : this is the construction exposed in Proposition  \ref{prop:imagemorphisme}.
It became a second subject in my PhD that I succeeded to solve in \cite{Durand:1998a} and that provided me many arguments to obtain an extension of Cobham's theorem for primitive substitutions (\cite{Durand:1998b}).

When he showed me this process I found curious to look at it. 
Why was it interesting~?
In fact, at this time, I did not know much about the work of G. Rauzy and almost nothing about interval exchange transformations (iets) and what he did with.
I realized quickly that this process was in fact the  induction process defined on dynamical systems by the Poincar\'e first return map : the derived sequence $\mathcal{D} (x)$ generates a subshift that is conjugate to the induced dynamical system $(X,T)$ on the cylinder set $[0]$, where $(X,T)$ is the subshift generated by $x$.
Thus, this corresponds to the induction process he defined \cite{Rauzy:1979} to study iets.
He induced on the left most interval and showed that the induced tranformation is again an iet with the same permutation.
This is also the case for primitive substitutions~: derived sequences of purely primitive substitutive sequences are purely primitive substitutive sequences (\cite{Durand:1998a}).
His goal was to tackle the Keane's conjecture saying that almost all iets are uniquely ergodic.
This was solved  independently in \cite{Masur:1982} and \cite{Veech:1982}.
Later M. Boshernitzan and C. R. Carroll \cite{Boshernitzan&Carroll:1997} proved that, for an iet defined on a quadratic field, "every consistent method of induction is eventually periodic" (up to rescaling).
This is what I obtained in \cite{Durand:1998a} in the context of primitive substitutions and that was expected by G. Rauzy.

Thus I imagine the intuition of G. Rauzy was certainly leaded by iet considerations.

I am very grateful to G\'erard Rauzy for the wonderful subject he gave me and for the opportunity he offered me to do research.

\subsection{The HD$0$L $\omega$-equivalence problem}

In this paper we propose to apply what the author developped in \cite{Durand:1998a} and \cite{Durand:1998b} (to answer the question of G. Rauzy) to solve the following problems in the primitive case.
To describe these problems we need some notations. 
Let $\sigma : A^* \to A^* $ and $\tau : B^* \to B^*$ be morphisms, and, $u$ and $v$ be words such that $\lim_{n\to \infty } \sigma^n (u)$ and $\lim_{n\to \infty } \tau^n (v)$ converge to respectively $x\in A^\mathbb{N}$ and $y\in B^\mathbb{N}$. 
We also need two morphisms $\phi : A^* \to C^*$ and $\psi : B^* \to D^*$.

\subsubsection{The HD$0$L $\omega$-equivalence problem}

This problem is open for more than 30 years and ask whether it is decidable to know if two morphic sequences are equal. 
More precisely, 

\begin{center}
Is the equality "$\phi (x) = \psi (y)$" decidable?
\end{center}

In this paper we answer positively to this problem for primitive morphisms.
The equality problem "$x=y$" (also called the D$0$L $\omega$-equivalence problem) was solved in 1984 by K. Culik II and T. Harju \cite{Culik&Harju:1984}.
An other proof was given by J. Honkala in \cite{Honkala:2007} (see also \cite{Honkala:2009a}).

\subsubsection{More on the equality of sequences generated by morphisms}

Once the equality "$\phi (x) = \psi (y)$" or "$x = y$" is satisfied, it is natural to look for relations between the involved morphisms.

In \cite{Seebold:1998} it is shown than when two invertible substitutions on a two letter alphabet share the same fixed point, then these substitutions are powers of some other morphism (for different proofs see also \cite{Berthe&Frettloh&Sirvent:preprint} and \cite{Rao&Wen:2010}).
For more general results on this topic see \cite{Krieger:2008,Diekert&Krieger:2009}. 

We will give a proof of a weaker result but that can be extended to bigger alphabets : if two primitive substitutions share the same fixed point and satisfy some conditions on their return words then they have a common power.

\subsubsection{The decidability of the periodicity of HD$0$L infinite sequences}

This is also an open problem for more than 30 years. 
It ask : 

\begin{center}
Is it decidable to know whether $\phi (x)$ is ultimately periodic ?
\end{center} 

J.-J. Pansiot proved in \cite{Pansiot:1986} that the periodicity of $x$ is decidable.
We answer positively to this problem in the primitive case and in a forthcoming paper the author should complete this work to close the problem.

In the sequel the vocabulary of D$0$L and HD$0$L systems will be replace by the vocabulary of substitutions.

\subsection{Some definitions}

\subsubsection{Words and sequences} 
An {\it alphabet} $A$ is a finite set of elements called {\it
  letters}. 
A {\it word} over $A$ is an element of the free monoid
generated by $A$, denoted by $A^*$. 
Let $x = x_0x_1 \cdots x_{n-1}$
(with $x_i\in A$, $0\leq i\leq n-1$) be a word, its {\it length} is
$n$ and is denoted by $|x|$. 
The {\it empty word} is denoted by $\epsilon$, $|\epsilon| = 0$. 
The set of non-empty words over $A$ is denoted by $A^+$. 
The elements of $A^{\mathbb{N}}$ are called {\it sequences}. 
If $x=x_0x_1\cdots$ is a sequence (with $x_i\in A$, $i\in \mathbb{N}$) and $I=[k,l]$ an interval of
$\mathbb{N}$ we set $x_I = x_k x_{k+1}\cdots x_{l}$ and we say that $x_{I}$
is a {\it factor} of $x$.  If $k = 0$, we say that $x_{I}$ is a {\it
 prefix} of $x$. 
The set of factors of length $n$ of $x$ is written
$\mathcal{L}_n(x)$ and the set of factors of $x$, or the {\it language} of $x$,
is denoted $\mathcal{L}(x)$. 
The {\it occurrences} in $x$ of a word $u$ are the
integers $i$ such that $x_{[i,i + |u| - 1]}= u$. 
If $u$ has an occurrence in $x$, we also say that $u$ {\em appears} in $x$.
When $x$ is a word,
we use the same terminology with similar definitions.

The sequence $x$ is {\it ultimately periodic} if there exist a word
$u$ and a non-empty word $v$ such that $x=uv^{\omega}$, where
$v^{\omega}= vvv\cdots $.
In this case $|v|$ is called a {\em length period} of $x$. 
It is {\it periodic} if $u$ is the empty word. 
A word $u$ is {\em recurrent} in $x$ if it appears in $x$ infinitely many times.
A sequence $x$ is {\it uniformly recurrent} if every factor $u$ of $x$
appears infinitely often in $x$ and the greatest
difference of two successive occurrences of $u$ is bounded.

\subsubsection{Morphisms and matrices} 
Let $A$ and $B$ be two alphabets. Let $\sigma$ be a {\it morphism} 
from $A^*$ to $B^*$. 
When $\sigma (A) = B$, we say $\sigma$ is a {\em coding}. 
Thus, codings are onto.
If $\sigma (A)$ is included in $B^+$, it induces by concatenation a map
from $A^{\mathbb{N}}$ to $B^{\mathbb{N}}$. These two maps are also called $\sigma$.
To the morphism $\sigma$ is naturally associated the matrix $M_{\sigma} =
(m_{i,j})_{i\in B , j \in A }$ where $m_{i,j}$ is the number of
occurrences of $i$ in the word $\sigma(j)$.

\subsubsection{Substitutions and substitutive sequences}

A {\it substitution} is a morphism $\sigma : A^* \rightarrow
A^{*}$. 
If there exists a letter $a\in A$ and a word $u\in A^+$ such that $\sigma(a)=au$ and moreover, if $\lim_{n\to+\infty}|\sigma^n(a)|=+\infty$, then $\sigma$ is said
    to be {\em prolongable on $a$}.

Since for all $n\in\mathbb{N}$, $\sigma^n(a)$ is a prefix of $\sigma^{n+1}(a)$ and because
     $|\sigma^n(a)|$ tends to infinity with $n$, the
     sequence $(\sigma^n(aaa\cdots ))_{n\ge 0}$ converges (for the usual
     product topology on $A^\mathbb{N}$) to a sequence denoted by
     $\sigma^\omega(a)$. 
The morphism $\sigma$ being continuous for the product topology, $\sigma^\omega(a)$ is a fixed point of $\sigma$: $\sigma (\sigma^\omega(a)) = \sigma^\omega(a)$.
A sequence obtained in
    this way by iterating a prolongable substitution is said to be {\em purely substitutive} (w.r.t. $\sigma$). 
If $x\in
    A^\mathbb{N}$ is purely substitutive and $\phi:A^*\to B^*$ is a morphism then the sequence $y=\phi (x)$ is said to be a {\em morphic sequence} (w.r.t. $\sigma$). 
When $\phi $ is a coding, we say $y$ is {\em substitutive} (w.r.t. $\sigma$). 

Whenever the matrix associated to $\sigma $ is
primitive (i.e., when it has a power with strictly positive coefficients) we say that $\sigma$ is a {\it primitive substitution}.  We
say $\sigma$ is
{\it erasing} if there exists $b\in A$ such that $\sigma (b)$ is the
empty word.

\section{Substitutions and return words}

The main results of the paper are consequences of results in \cite{Durand:1998a} (see also \cite{Durand:2000}) taking care to be explicit on some bounds (Section \ref{subsec:returnwords}). 
Since we are dealing with decidability problems, it is important to pay attention that these bounds are algorithmically computable because they are involved in the decidability problems we are dealing with.

\subsection{Useful results on return words of morphic sequences}
\label{subsec:returnwords}

Let $x$ be a uniformly recurrent sequence on the alphabet $A$ and $u$ a non-empty prefix of $x$. We call {\it return word to u} every word
$x_{[i,j-1]} = x_ix_{i+1} \cdots x_{j-1}$, where $i$ and $j$ are two successive occurrences of $u$ in $x$. 
The reader can check that a word $v$ is a return word to $u$ of $x$ if and only if $vu$ belongs to $\mathcal{L}(x) = \{ x_{[i,j]} ; 0\leq i\leq j \}$, $u$ is a prefix of $vu$ and $u$ has exactly two occurrences in $vu$. 
The set of return words to $u$ is finite, because $x$ is uniformly recurrent, and is denoted by ${\mathcal{R}}_{x,u}$. 
The sequence $x$ can be written as a concatenation
$$
x = m_0m_1m_2 \cdots \:\: , m_i \in \mathcal{R}_{x,u}, i\in \mathbb{N},
$$
of return words to $u$, and this decomposition is unique. We enumerate the elements of $\mathcal{R}_{x,u}$ in the order of their first appearance in $(m_n)_{n\in\mathbb{N}}$. 
This defines a bijective map
$$
\Theta_{x,u} : R_{x,u}  \rightarrow {\mathcal{R}}_{x,u} \subset A^{*}
$$

where $R_{x,u} = \{ 1, \cdots , \#({\mathcal{R}}_{x,u})\}$. 
It defines a morphism and the set $\Theta_{x,u} (R_{x,u} ^*)$ consists of all concatenations of return words to $u$.
We should point out that the order used to define this morphism will be of particular importance in the sequel. 
In fact, you can change the enumeration and this will change the morphism. 
What is important in the sequel is that once you fixed an algorithm to define $\Theta_{x,u}$ then you should use the same for all $x$ and $u$.

We denote by $\mathcal{D}_u(x)$ the unique sequence on the alphabet $R_{x,u}$ characterized by
$$
\Theta_{x,u} (\mathcal{D}_u(x) ) = x.
$$

We call it the {\it derived sequence of $x$ on $u$}. It is clearly uniformly recurrent. 
The following proposition points out the basic properties of return words.

\begin{prop}[\cite{Durand:1998a}]
\label{prop:derder}
Let $x$ be a uniformly recurrent sequence and $u$ a non-empty prefix of $x$.
\begin{enumerate}
\item 
The set $\mathcal{R}_{x,u}$ is a code, i.e., $\Theta_{x,u} : R_{x,u}^* \rightarrow \Theta_{x,u} (R_{x,u} ^*)$ is one-to-one. 
\item 
If $u$ and $v$ are two prefixes of $x$ such that $u$ is a prefix of $ v$ then each return word to $v$ belongs to $\Theta_{x,u} (R_{x,u} ^*)$, i.e., it is a concatenation of return words to $u$.
\item 
Let $v$ be a non-empty prefix of $\mathcal{D}_u (x)$ and $w = \Theta_{x,u}(v)u$. Then 
\begin{itemize}
\item 
$w$ is a  prefix $x$,
\item $\Theta_{x,u}\Theta_{\mathcal{D}_u(x),v} = \Theta_{x,w}$ and
\item $\mathcal{D}_v (\mathcal{D}_u (x)) = \mathcal{D}_w(x)$. 
\end{itemize}
\end{enumerate}
\end{prop}

\begin{coro}
\label{coro:theta}
Let $x$ be a uniformly recurrent sequence and $u$ a non-empty prefix of $x$.
If $u$ is a prefix of $v$ and $v$ is a prefix of $x$ then there is a unique morphism $\Theta $ satisfying $\Theta_{x,u} \Theta = \Theta_{x,v}$.
\end{coro}

The next lemma asserts that the derived sequences of purely primitive substitutive sequences are purely primitive substitutive sequences.

\begin{prop}[\cite{Durand:1998a}]
\label{prop:derivedsub}
Let $x=\sigma^\omega (a)$ where $\sigma : A^* \to A^*$ is a primitive substitution, and, $u$ be
a non-empty prefix of $x$. 
The derived sequence $\mathcal{D}_u (x) $ is the unique
fixed point of the unique substitution $\sigma_{u}$ satisfying

$$
\Theta_{x,u} \sigma_u = \sigma \Theta_{x,u}.
$$  
\end{prop}

The substitution $\sigma_u$ is called the {\em return substitution} on $u$.

Remark that it is easy to compute $\sigma_u$ and $\Theta_{x,u}$.
Take $w$ the first return word to $u$ having an occurrence in $x$ : $\Theta_{x,u} (1) = w$. 
Then, compute $\sigma (w)$. 
It is necessarily a (computable) concatenation of return words to $u$ and a prefix of $x$ : $\sigma (w) = \Theta_{x,u} (v_1)$ for some $v_1\in R_{x,u}^*$. 
This defines $\sigma_u (1) = v_1$.
If no return words different from $w$ appears then we stop here and $x$ is periodic : $x=www\dots $.
If not, we do the same with $\Theta_{x,u} (2)$ and we continue until no new return word appears. 
This will stop in finite time because $x$ is ultimately recurrent. 
Remark that this procedure also defines the map $\Theta_{x,u}$.
We proved the following lemma.

\begin{lemma}
\label{lemma:returnsubalgo}
Let $x=\sigma^\omega (a)$ where $\sigma : A^* \to A^*$ is a primitive substitution.
Let $u$ be a prefix of $x$.
Then, $\sigma_u$ and $\Theta_{x,u}$ are algorithmically computable. 
\end{lemma}

Bounds will be given with the help of the next 3 lemmas.
For any morphism $\sigma : A^* \to B^*$, we set $|\sigma | = \max_{a\in A} |\sigma |$.

In the sequel $R_\sigma$ will denote the maximal difference between two successive occurrences of a word of length $2$ in any fixed point of the primitive substitution $\sigma$.
In order to bound $R_\sigma$ we need the following lemma.

\begin{lemma}
\label{lemme:horn}
\cite{Horn&Johnson:1990}
If $M$ is a primitive $n\times n$ matrix then there exists $k\leq (n-1)n^n$ such that $M^k $ has strictly positive entries.
\end{lemma}

\begin{lemma}
If $\sigma $ is a primitive substitution defined on an alphabet with $d$ letters, then $R_\sigma$ is bounded by $2|\sigma |^{(d-1)d^d}$.
\end{lemma}

\begin{proof}
The proof is left to the reader.
\end{proof}

\begin{lemma}
For all primitive substitution defined on an alphabet with $d$ letters, then
for all $n$

$$
|\sigma^n | \leq Q_\sigma \min_{a\in A} |\sigma^n (a) | .
$$

where $Q_\sigma = \max\left( \max_{0\leq n\leq (d-1)d^d +1} \frac{|\sigma^n |}{\min_a |\sigma^n (a)|} , |\sigma |\right)$.

\end{lemma}

\begin{proof}
The proof is left to the reader.
\end{proof}

Now we can give upper and lower bounds for the length of the return words and for the cardinality of the sets of return words.

\begin{theo}[\cite{Durand:1998a,Durand&Host&Skau:1999}]
\label{theo:encad}
Suppose $x=\sigma^\omega (a)$, where $\sigma$ is a primitive substitution, is non-periodic. 
For all non-empty prefixes $u$ of $x$, 
\begin{enumerate}
\item 
\label{theo:item:1}
for all words $v$ belonging to $\mathcal{R}_{x,u}$, $\frac{1}{K_\sigma}| u| \leq | v| \leq K_\sigma  | u|$, and 
\item 
\label{theo:item:2}
$\# (\mathcal{R}_{x,u})\leq 4K_\sigma^3$,
\end{enumerate}

where $K_\sigma = Q_\sigma R_\sigma |\sigma|$.
\end{theo}

\begin{proof}
The upper bound in \eqref{theo:item:1} is easily deduced from \cite{Durand:1998a}.
The lower bound can also be deduced from \cite{Durand:1998a} but it is easier to obtain from \cite{Durand&Host&Skau:1999}.
The last bound can be deduced from \eqref{theo:item:1}. 
\end{proof}

Note that the right inequality of Statement \eqref{theo:item:1} is sufficient to imply the two other inequalities (see \cite{Durand&Host&Skau:1999}).
From this theorem and the definition of $\sigma_u$ we deduce the following corollary.

\begin{coro}
For all primitive substitution $\sigma$ and all $u$ such that $u$ is a prefix of $\sigma (u)$ we have

$$
| \sigma_u | \leq |\sigma |K_\sigma^2.
$$

\end{coro}

The result we stated above were used to prove the two following results that will be involved in the decidability problems.
The first will be important in the purely substitutive (or D$0$L) context and the second in the morphic (or HD$0$L) case.

\begin{prop}[\cite{Durand:1998a}]
\label{prop:cardsub}
Let $\sigma$ be a primitive substitution. 
The set of the return substitutions of $\sigma$ is finite and bounded by $\left(4K_\sigma^3\right)^{|\sigma | K_\sigma^2 +1}$.
\end{prop}

\begin{prop}
\label{prop:deflambda}
Let $x=\sigma^\omega (a)$, where $\sigma$ is a primitive substitution, and $\phi : A^* \to B^*$ be a coding.
For all prefix $u$ of $x$ there exists a unique morphism $\lambda_u$ satisfying

$$
\phi \Theta_{x,u} =  \Theta_{\phi (x), \phi (u)} \lambda_u.  
$$

Moreover, 

\begin{enumerate}
\item
$\lambda_u ( \mathcal{D}_u (x)) = \mathcal{D}_{\phi (u)} (\phi (x))$,
\item
$\# \{ \lambda_u ; u \hbox{ prefix of } x \} \leq (K_\sigma +1)^{K_\sigma^2}$.
\end{enumerate}

\end{prop}

\begin{prop}
\label{prop:imagereturnalgo}
With the settings of Proposition \ref{prop:deflambda},
let $u$ be a prefix of $x$.
Then, $\Theta_{\phi (x) , \phi (u)}$ and $\lambda_u$ can be algorithmically defined. 
\end{prop}

\section{HD$0$L $\omega$-equivalence problem}

\subsection{Decidability results for purely substitutive sequences in a primitive context}

In this section we propose a strategy to decide whether two purely substitutive sequences are equal when the underlying substitutions are primitive. 

The following lemma is easy to establish (see \cite{Durand&Host&Skau:1999}).
In the sequel the constants $K_\sigma$ and $K_\tau$ are those given by Theorem \ref{theo:encad}.

\begin{lemma}
\label{lemma:DHS}
Let $x$ be a uniformly recurrent sequence such that there exists $K$ satisfying : for all $u\in \mathcal{L}(x)$ and all $w\in \mathcal{R}_{x,u}$, $|w|\leq K |u|$. 
Then, for all $n$, all words of length $(K+1)n$ of $\mathcal{L}(x)$ have an occurrence of all words of length $n$. 
\end{lemma}

\begin{theo}
\label{theo:decid}
Let $\sigma : A^* \to A^* $ and $\tau : B^*\to B^*$ be primitive substitutions, $x=\sigma^\omega (a)\in A^\mathbb{N}$ and $y=\sigma^\omega (b)\in B^\mathbb{N}$.
Let $K =1+ \left((4K_\sigma )^3\right)^{|\sigma | K_\sigma^2 +1}\left((4K_\tau )^3\right)^{|\tau | K_\tau^2 +1}$, $u_n = x[0,(K_\sigma +1)^n]$, $v_n = y[0,(K_\sigma +1)^n]$ and $T_n = ( \sigma_{u_n}, \tau_{v_n} )$. 
Then, the following are equivalent :

\begin{enumerate}
\item
$x=y$;
\item
there exist $n$ and $m$ less than $K$ such that 
$$
u_n = v_n, u_m=v_m, \Theta_{x,u_n} = \Theta_{y,v_n} , \Theta_{x,u_m} = \Theta_{y,v_m} \hbox{ and } T_n = T_m. 
$$
\end{enumerate}

\end{theo}

\begin{proof} 
The necessary condition is obtained applying Proposition \ref{prop:cardsub} and the pigeon hole principle.
 
Let us show it is sufficient : suppose there exist $n,m$ satisfying $n<m\leq K$ such that $u_n = v_n$, $u_m=v_m$,  $\Theta_{x,u_n} = \Theta_{y,v_n}$, $\Theta_{x,u_m} = \Theta_{y,v_m}$ and $T_n = T_m$.

From Proposition \ref{prop:derivedsub} we deduce that $\mathcal{D}_{u_n} (x) = \mathcal{D}_{u_m} (x)$ and $\mathcal{D}_{u_n} (y) = \mathcal{D}_{u_m} (y)$.
Moreover, from Corollary \ref{coro:theta} there exists a unique morphism $\Theta : R^*\to R^*$ such that $\Theta_{x,u_n} \Theta = \Theta_{x,u_m}$ where $R=R_{x,u_n}=R_{x,u_m}=R_{y,u_n}=R_{y,u_m}$.
Then, 

\begin{align*}
\Theta_{x,u_n} \Theta (\mathcal{D}_{u_m} (x)) & = \Theta_{x,u_m} (\mathcal{D}_{u_m} (x)) = x \hbox{ and }\\
\Theta_{x,u_n} \Theta (\mathcal{D}_{u_m} (y)) & = \Theta_{x,u_m} (\mathcal{D}_{u_m} (y)) = y .
\end{align*}

Thus, from the definition of $\mathcal{D}_{u_n} (x)$ and $\mathcal{D}_{u_n} (y)$ we deduce :

\begin{align*}
\Theta (\mathcal{D}_{u_m} (x)) & = \mathcal{D}_{u_n} (x)= \mathcal{D}_{u_m} (x)\\
\Theta (\mathcal{D}_{u_m} (y)) & = \mathcal{D}_{u_n} (y)= \mathcal{D}_{u_m} (y)
\end{align*}

and $\mathcal{D}_{u_n} (x)$ and $\mathcal{D}_{u_n} (y)$ are fixed points of $\Theta$.
But remark that from Lemma \ref{lemma:DHS} and the choice of the $u_n$, in all images of $\Theta$ there are all the letters of $R$ and all images of $\Theta$ start with the same letter. 
Hence $\Theta$ is a primitive substitution and $\Theta$ has a unique fixed point. 
Consequently, $\mathcal{D}_{u_n} (x) = \mathcal{D}_{u_n} (y)$ and $x=\Theta_{x,u_n} (\mathcal{D}_{u_n} (x))=\Theta_n (\mathcal{D}_{u_n} (y)) =y$. 
\end{proof}

\begin{coro}
The D$0$L $\omega$-equivalence problem is decidable in the primitive case.
\end{coro}

\begin{proof}
This is a direct consequence of Theorem \ref{theo:decid}.
Let us keep the notation of this theorem.
The bounds $K_\sigma$, $K_\tau$ and $K$ are algorithmically computable. 
So are the $u_n$ and $v_n$, $n\leq K$.
We gave the algorithm to compute return substitutions after their definition, hence $T_n$ is algorithmically computable and the D$0$L $\omega$-equivalence problem is decidable.
\end{proof}

\subsection{Decidability results for morphic sequences in the primitive context}

The strategy to tackle the decidability of the HD$0$L $\omega$-equivalence problem in the primitive case is the same as for the D$0$L case.

\begin{theo}
\label{theo:hdoldecid}
Let $\sigma : A^* \to A^* $ and $\tau : B^*\to B^*$ be two primitive substitutions, $x=\sigma^\omega (a)\in A^\mathbb{N}$ and $y=\sigma^\omega (b)\in B^\mathbb{N}$.
Let $\phi : A^* \to C^*$ and $\psi : B^* \to D^*$ be two codings.
Let $K =1+ \left((4K_\sigma )^3\right)^{|\sigma | K_\sigma^2 +1} (K_\sigma +1)^{K_\sigma^2}\left((4K_\tau )^3\right)^{|\tau | K_\tau^2 +1} (K_\tau +1)^{K_\tau^2}$, $u_n = \phi (x[0,(K_\sigma +1)^n])$, $v_n = \psi (y[0,(K_\sigma +1)^n])$ and $T_n = (\lambda_{u_n},\sigma_{u_n}, \lambda_{v_n}, \tau_{v_n} )$. 
Then, the following are equivalent :

\begin{enumerate}
\item
$x=y$;
\item
there exist $n$ and $m$ less than $K$ such that 
$$
u_n = v_n, u_m=v_m, \Theta_{x,u_n} = \Theta_{y,u_n} , \Theta_{x,u_m} = \Theta_{y,u_m} \hbox{ and } T_n = T_m. 
$$
\end{enumerate}

\end{theo}

\begin{proof}
The proof proceed as the proof of Theorem \ref{theo:decid}.
\end{proof}

The following proposition will allow us to reduce the HD$0$L $\omega$-equivalence problem to substitutive sequences. 
This result can also be found in \cite{Cobham:1968,Pansiot:1983,Cassaigne&Nicolas:2003,Durand:1998a,Honkala:2009b}.
Below I present the proof G. Rauzy showed me.

\begin{prop}
\label{prop:imagemorphisme}
Let $x$ be a morphic sequence w.r.t. a primitive substitution.
Then, $x$ is substitutive w.r.t. a primitive substitution. 
\end{prop}

\begin{proof}
Let $y$ be a purely substitutive sequence generated by the primitive substitution $\sigma : A^* \to A^*$ and $\rho : A^* \to B^*$ be a morphism such that $x = \rho (y)$. 
Suppose $y=\sigma^\omega (a)$.

From Lemma \ref{lemme:horn}, there exists $k \leq (m-1)^m$, where $m=\# A$, such that all images of $\sigma^k$
have an occurrence of all letters of $A$.
Thus, the morphism $\phi = \rho \sigma^k$ does not have the empty word as an image of letter.
Moreover, $x=\phi (y)$.

Let $D = \{(c,k) ; c\in C , 0\leq k \leq | \phi (c)|-1 \}$ 
and $\psi : C^* \rightarrow D^* $ the morphism defined by:

\[
\psi (c) = (c,0) \ldots
(c,| \phi (c)|-1 ).
\]

There is an integer $n$ ($\leq \max_{a\in A} |\phi (a)|$) such that $| \sigma^n (c)|  \geq| \phi (c)| $ 
for all $c$ in $C$.

Let $\tau $ be the morphism from $D $ to $ D^* $ defined by:
\begin{center}
\begin{tabular}{ll}
   &$\tau((c,k))=\psi (\sigma^n (c)_{[k,k]})$ \ if $0\leq k<|\phi(c)|-1$,\\  
and &$\tau((c,|\phi(c)|-1))=\psi(\sigma^n(c)_{[|\phi(c)|-1,|\sigma^n(c)|-1]})$
 \ otherwise.\\
\end{tabular}\\
\end{center}

For all $c$ in $C$ we have
$$
\begin{array}{rl}
\tau (\psi (c)) &=\tau ((c,0)\cdots (c,| \phi (c) |-1 ))\\
                &\\
                &=\psi (\sigma^n (c)_{[0,0]} )\cdots \psi (\sigma^n 
                     (c)_{[| \sigma^n (c) |-1,| \sigma^n (c) |-1]})
 = \psi (\sigma^n (c)),
\end{array}
$$
hence $ \tau (\psi (y))=\psi (\sigma^n (y))=\psi(y )$. 
Hence
$\psi (y)$ is a fixed point of the substitution $\tau$ which 
begins with some letter $(e,0)$ and $\tau \psi =\psi \sigma^n$. 

Let $\chi $ be the coding from $D^*$ to $ B^*$ defined 
by $\chi ((c,k)) = \phi(c)_{[k,k]}$ for all $(c,k)$ in $D$. For all 
$c$ in $C$ we obtain
\[
\chi (\psi (c))
= \chi ((c,0) \cdots (c,| \phi (c) | - 1))= \phi (c),
\]
and then $\chi (\psi (y)) = \phi (y)$.
Moreover, the relation $\tau \psi =\psi \sigma^n$ implies that for all $k \in \mathbb{N}$ 
we have $\tau^k \psi =\psi \sigma^{kn}$. Hence if $\sigma$ is primitive 
then $\tau$ is primitive too.
\end{proof}

\begin{coro}
The HD$0$L-$\omega$-equivalence problem is decidable in the primitive case.
\end{coro}

\begin{proof}
It is a consequence of Proposition \ref{prop:imagemorphisme} and Theorem \ref{theo:hdoldecid}.
\end{proof}

\section{More on substitutions sharing the same fixed point}

A substitution on a finite alphabet is {\em invertible} if when extended to a free group, it is an automorphism. 
Let $x$ be a sequence on the finite alphabet $A$ and suppose all letters of $A$ have an occurrence in $x$.
The {\em stabilizer} of $x$, denoted by ${\rm Stab} (x)$, is the set of morphisms $\sigma : A^* \to A^*$ such that $\sigma (x) = x$.
We say that $x$ is {\em rigid} whenever the monoid ${\rm Stab} (x)$ is generated by a single element (see \cite{Krieger:2008} for more informations on this topic).  
It has been shown in \cite{Seebold:1998} that all invertible substitutions generate rigid sequences (Theorem \ref{theo:seebold}).
There was in fact a gap in the proof. 
In \cite{Rao&Wen:2010} the authors gave a geometrical proof (see also \cite{Berthe&Frettloh&Sirvent:preprint} for a different proof). 
The gap in the proof in \cite{Seebold:1998} was recently corrected in \cite{Richomme&Seebold:preprint}.
The statement of this result is the following.

\begin{theo}
\label{theo:seebold}
Let $\sigma$ and $\tau$ be primitive invertible substitutions defined on the alphabet $\{ a,b\}$.
They have powers sharing the same fixed point if and only if there exist a substitution $\zeta$ and positive integers $i$ and $j$ such that $\sigma = \zeta^i$ and $\zeta^j = \tau$.
\end{theo}

As pointed out in \cite{Krieger:2008}, the invertibility assumption is crucial in this theorem. 
The following example due to S\'e\'ebold (see \cite{Krieger:2008}) gives a counterexample :

$$
\sigma : a\mapsto ab, \ b\mapsto baabba , \ \tau : a\mapsto abbaab , \ b\mapsto ba .
$$

Indeed, $\sigma^\omega (a) = \tau^\omega (a)$ but $\sigma$ and $\tau$ are not powers of the same substitution.

For a 3-letter alphabet, it has been shown in \cite{Wen&Zhang:1999} (see also \cite{Richomme:2003})
that the monoid of invertible substitutions is not finitely generated.
This leaded D. Krieger \cite{Krieger:2008} to find a non rigid purely substitutive sequence $x$ w.r.t. an invertible substitution.

Below we propose a weaker version of Theorem \ref{theo:seebold} (Corollary \ref{coro:seebold}) that applies to any kind of alphabets.
Before we need a lemma, a proposition and the following notation : for any word $w$, $V(w)$ is the vector whose entry of index $a$ is the number of $a$ in $w$. 
It is sometimes called the {\em Parikh vector of $w$} or the {\em abelianized} of $w$.

\begin{lemma}
\label{lemma:1998}
Let $x=\sigma^\omega (a)$ and $y=\tau^\omega (a)$ where $\sigma$ and $\tau$ are two primitive substitutions (on an alphabet with $d$ letters). 
Suppose $x$ and $y$ have a common prefix $u$ satisfying :

\begin{enumerate}
\item
$\sigma (w) =\tau (w)$ for all $w\in \mathcal{R}_{x,u}$;
\item
$\{ V(w) ; w\in \mathcal{R}_{x,u} \}$ generates $\mathbb{Z}^d$. 
\end{enumerate} 

Then $\tau = \sigma$.
\end{lemma}

\begin{proof}
The proof is left to the reader.
\end{proof}
\begin{prop}
\label{prop:1998}
\cite{Durand:1998b}
Let $\tau$ and $\sigma$ be two primitive substitutions.
Then, they share the same non-periodic fixed point $\sigma^\omega (a) = \tau^\omega (a)$ if and only if  $\sigma^\omega (a)$ and  $\tau^\omega (a)$ share a common prefix $u$ such that there exist two integers $i$ and $j$ 
verifying 

\begin{align}
\label{egalite:1998}
\tau_u^i = \sigma_u^j.
\end{align}
\end{prop}

Remark that Equality \eqref{egalite:1998} means that $\tau^i$ and $\sigma^j$ coincide on the set $\mathcal{R}_{x,u}$.
Thus, using Lemma \ref{lemma:1998} we obtain the announced corollary.

\begin{coro}
\label{coro:seebold}
Let $\tau$ and $\sigma$ be two primitive substitutions (on an alphabet with $d$ letters) sharing the same non-periodic fixed point $x= \sigma^\omega (a) = \tau^\omega (a)$.
Let $i$, $j$ be the positive integers and $u$ the prefix of $x$ given in Proposition \ref{prop:1998}.
Suppose the set $\{ V(w) ; w\in \mathcal{R}_{x,u} \}$ generates $\mathbb{Z}^d$. 
Then, there exist two integers $i$ and $j$ such that
$$
\tau^i = \sigma^j.
$$
\end{coro}

Applied to invertible substitutions on a 2-letter alphabet this corollary implies that when two invertible primitive substitutions share the same fixed point then they have non trivial common powers because it can be checked that for all $u$ the set $\{ V(w) ; w\in \mathcal{R}_{x,u} \}$ generates $\mathbb{Z}^2$. 
For the counterexample of Seebold recalled above this means that all $u$ satisfying Equality \eqref{egalite:1998} is such that all vectors in $\{ V(w) ; w\in \mathcal{R}_{x,u} \}$ are $\mathbb{Q}$-colinear.

Proposition \ref{prop:1998} provides a way to study rigity problems : with a weaker assumption than the one in Corollary \ref{coro:seebold} but with specific families of substitutions it is clear that the conclusion of this corollary would also hold. 

To finish with this section we would like to mention the following complete substitutive extension of Cobham's theorem  obtained in \cite{Durand:2011}. 
This gives a strong restriction (without the primitivity assumption) to generate the same morphic sequence.

\begin{theo}
\label{theo:2011}
Let $\alpha$ and $\beta$ be two multiplicatively independent Perron numbers. Let $A$ be a finite alphabet and $x$ be a sequence of $A^{\mathbb{N}}$. Then, 
$x$ is both $\alpha$-substitutive and $\beta$-substitutive
if and only if
$x$ is ultimately periodic,
\end{theo}

where "$x$ is $\alpha$-substitutive" means that $x$ is substitutive w.r.t. a substitution whose dominant eigenvalue of its incidence matrix is $\alpha$.
This result also holds true for morphic sequences whenever the morphisms (instead of the codings) are non-erasing.
For purely substitutive substitutions it (the necessary condition) was known from \cite{Durand:2002}.
  
\section{Periodicity of HD0L infinite sequences}

In the sequel $\sigma : A^*\to A^*$ is a primitive morphism, $\phi : A^* \to B^*$ is a morphism, $y=\sigma^\omega  (a)$ is a sequence of  $A^\mathbb{N}$ and $x=\phi (y)$.
We have to find an algorithm telling if $x$ is ultimately periodic or not.
Remark that as the morphism is primitive, the sequence $x$ is uniformly recurrent, and, ultimate periodicity of $x$ is equivalent to its periodicity. 
The following lemma is easy to prove.

\begin{lemma}
\label{lemma:period-return}
For any periodic sequence $z$ with length period $L$, if $u$ is a word of $\mathcal{L}(z)$ whose length is greater than $L$ then $\# 
\mathcal{R}_u =1$. 
\end{lemma}

We recall (Proposition \ref{prop:derder}) that for all $n$ there exists a prefix $u_n$ of $y$ such that $\mathcal{D}^n (x) = \lambda_{n}(\mathcal{D}_{u_n} (y))$ and $u_1 = x_0$ where $\lambda_n$ is the unique morphism from $R_{u_n}^*$ to $R_{\phi(u_n)}^*$ defined by 

$$
\phi \Theta_{x,u_n} =  \Theta_{\phi (x), \phi (u_n)} \lambda_n.  
$$

We set $\sigma_n = \sigma_{u_n}$.

\begin{prop}
\label{prop:decid-per}
Suppose $\phi$ is a coding and the incidence matrix of $\sigma$ has strictly positive entries.
Let $K = \left(4K_\sigma^3\right)^{|\sigma | K_\sigma^2 +1}(K_\sigma +1)^{K_\sigma^2}$.
Then, $x$ is periodic if and only if there exists $i$ less than $K$ such that $\# \mathcal{R}_{\phi (u_i)} = 1$.
\end{prop}

\begin{proof}
If $\# \mathcal{R}_{\phi (u_i)} = 1$, then it is clear that $x$ is periodic.

Now suppose $x$ is periodic.
From the choice of $K$ and the pigeon hole principle, there exist $i$ and $j$, $i<j$, such that
$(\lambda_i , \sigma_i) = (\sigma_j , \lambda_j)$.
Thus, $R_{u_i} = R_{u_j}$ and $R_{\phi(u_i)} = R_{\phi (u_j)}$.
Moreover, remark that we have $\mathcal{D}^i (x) = \mathcal{D}^j (x)$. 
Thus, $\mathcal{D}^i (x) = \mathcal{D}^{i+k(j-i)} (x)$ for all $k\geq 1$.
From Proposition \ref{prop:derder}, $\lim_{k\to +\infty } |u_{i+k(j-i)}| = +\infty $.
We conclude with Lemma \ref{lemma:period-return}.
\end{proof}

\begin{theo}
The HD0L periodicity problem for infinite sequences is decidable in the context of primitive morphisms.
\end{theo}

\begin{proof}
First, let us make some reductions of the problem.
From Lemma \ref{lemme:horn}, we can suppose the incidence matrix of $\sigma$ has strictly positive entries.
Suppose $\phi$ is erasing. 
It cannot erase all the letters.
Hence $\psi = \phi\sigma$ is a non erasing morphism and remark that $x=\psi (y)$.
Consequently, we can suppose $\phi$ non-erasing.
To end with the reductions, from Proposition \ref{prop:imagemorphisme} we can suppose $\phi $ is a coding.

Now, with the notations introduced in this section, it suffices to compute $R_{\phi (u_n)}$ from $n=1$ to $n=K$.
If for some $i$ less than $K$, $\# R_{\phi (u_n)} = 1$ then $x$ is periodic, otherwise $x$ is not periodic.
It remains to explain how to compute $R_{\phi (u_n)} $.
We recall the maps $\lambda_{u_n} : R_{y, u_n} \to R_{x, \phi (u_n)}$ are algorithmically definable (see Proposition \ref{prop:imagereturnalgo}). 
Hence it can be easily check that $\# R_{\phi (u_n)} = 1$. 
\end{proof}

\bibliographystyle{new2}
\bibliography{honkala}

\end{document}